\documentclass[11pt]{article}
\usepackage{amsmath,amssymb,amsthm,amscd}

\newtheorem{theorem}{Theorem}[section]
\newtheorem{lemma}[theorem]{Lemma}
\newtheorem{proposition}[theorem]{Proposition}

\theoremstyle{plain}

\theoremstyle{definition}

\newtheorem{remark}[theorem]{Remark}

\numberwithin{equation}{section}

\renewcommand{\labelenumi}{\textup{(\theenumi)}}

\newcommand{\Homeo}{\operatorname{Homeo}}

\newcommand{\id}{\operatorname{id}}
\newcommand{\Ker}{\operatorname{Ker}}
\newcommand{\sgn}{\operatorname{sgn}}
\newcommand{\Ad}{\operatorname{Ad}}

\newcommand{\Z}{\mathbb{Z}}
\newcommand{\T}{\mathbb{T}}

\def\R{{\mathcal{R}}}
\def\WRA{{\widetilde{\R}_A}}

\title{ $C^*$-algebras associated with  asymptotic equivalence relations defined by
 hyperbolic toral automorphisms }
\author{Kengo Matsumoto \\
Department of Mathematics \\
Joetsu University of Education \\
Joetsu, 943-8512, Japan
}

\begin{document}
\maketitle

\date{}

\def\det{{{\operatorname{det}}}}

%\maketitle
\begin{abstract}
We study the $C^*$-algebras of the \'etale groupoids defined by 
 the asymptotic equivalence relations for hyperbolic automorphisms on the two-dimensional torus.
The algebras are proved to be four-dimensional non-commutative tori
by an explicit numerical computation. 
The ranges of the unique tracial states of its $K_0$-groups of the $C^*$-algebras
 are described in terms of the hyperbolic matrix of the automorphism on the torus.
\end{abstract}

2020{\it Mathematics Subject Classification}:
 Primary 37D20, 37A55; Secondary 46L35.


%%%%%%%%%%%%%%%%%%%%%%%%%%%%%%%%%%%%%%%%%%%%%%%%%%%%   
\def\R{{\mathcal{R}}}
\def\OA{{{\mathcal{O}}_A}}
\def\OB{{{\mathcal{O}}_B}}
\def\RA{{{\mathcal{R}}_A}}
\def\WRA{{\widetilde{\R}_A}}
\def\WRB{{\widetilde{\R}_B}}
\def\WRC{{\widetilde{\R}_C}}
\def\RAC{{{\mathcal{R}}_A^\circ}}
\def\RAR{{{\mathcal{R}}_A^\rho}}
\def\RTA{{{\mathcal{R}}_{A^t}}}
\def\RB{{{\mathcal{R}}_B}}
\def\FA{{{\mathcal{F}}_A}}
\def\FB{{{\mathcal{F}}_B}}
\def\FA{{{\mathcal{F}}_A}}
\def\FTA{{{\mathcal{F}}_{A^t}}}
\def\FB{{{\mathcal{F}}_B}}
\def\FTB{{{\mathcal{F}}_{B^t}}}
\def\DTA{{{\mathcal{D}}_{{}^t\!A}}}
\def\FDA{{{\frak{D}}_A}}
\def\FDTA{{{\frak{D}}_{{}^t\!A}}}
\def\FFA{{{\frak{F}}_A}}
\def\FFTA{{{\frak{F}}_{A^t}}}
\def\FFB{{{\frak{F}}_B}}
\def\OZ{{{\mathcal{O}}_Z}}
\def\V{{\mathcal{V}}}
\def\E{{\mathcal{E}}}
\def\G{{\mathcal{G}}}
\def\OTA{{{\mathcal{O}}_{A^t}}}
\def\OTB{{{\mathcal{O}}_{B^t}}}
\def\SOA{{{\mathcal{O}}_A}\otimes{\mathcal{K}}}
\def\SOB{{{\mathcal{O}}_B}\otimes{\mathcal{K}}}
\def\SOZ{{{\mathcal{O}}_Z}\otimes{\mathcal{K}}}
\def\SOTA{{{\mathcal{O}}_{A^t}\otimes{\mathcal{K}}}}
\def\SOTB{{{\mathcal{O}}_{B^t}\otimes{\mathcal{K}}}}
\def\DA{{{\mathcal{D}}_A}}
\def\DB{{{\mathcal{D}}_B}}
\def\DZ{{{\mathcal{D}}_Z}}

\def\SDA{{{\mathcal{D}}_A}\otimes{\mathcal{C}}}
\def\SDB{{{\mathcal{D}}_B}\otimes{\mathcal{C}}}
\def\SDZ{{{\mathcal{D}}_Z}\otimes{\mathcal{C}}}
\def\SDTA{{{\mathcal{D}}_{A^t}\otimes{\mathcal{C}}}}
\def\O2{{{\mathcal{O}}_2}}
\def\D2{{{\mathcal{D}}_2}}

%%%%%%%%%%%%%%%%%%%%%%%%%%%%%%%
\def\Max{{{\operatorname{Max}}}}
\def\Tor{{{\operatorname{Tor}}}}
\def\Ext{{{\operatorname{Ext}}}}
\def\Per{{{\operatorname{Per}}}}
\def\PerB{{{\operatorname{PerB}}}}
\def\Homeo{{{\operatorname{Homeo}}}}
\def\Out{{{\operatorname{Out}}}}
\def\Aut{{{\operatorname{Aut}}}}
\def\Ad{{{\operatorname{Ad}}}}
\def\Inn{{{\operatorname{Inn}}}}
\def\det{{{\operatorname{det}}}}
\def\exp{{{\operatorname{exp}}}}
\def\cobdy{{{\operatorname{cobdy}}}}
\def\Ker{{{\operatorname{Ker}}}}
\def\ind{{{\operatorname{ind}}}}
\def\id{{{\operatorname{id}}}}
\def\supp{{{\operatorname{supp}}}}
\def\co{{{\operatorname{co}}}}
\def\Sco{{{\operatorname{Sco}}}}
\def\Spm{{\operatorname{Sp}}_{\operatorname{m}}^\times}
\def\U{{{\mathcal{U}}}}
%%%%%%%%%%%%%%%%%%%%%%%%%%
\def\GL{{{\operatorname{GL}}}}
\def\mod{{{\operatorname{mod}}}}

%%%%%%%%%%%%
%%%%%%%%%%%%%%%%%%%%%%%%%
\section{Introduction}
%%%%%%%%%%%%%%%%%%%
In \cite{Ruelle1} and \cite{Ruelle2}, D. Ruelle has introduced the notion of Smale space.
A Smale space is a hyperbolic dynamical system with local product structure.
He has constructed groupoids and its operator algebras from the Smale spaces.
After the Ruelle's initial study,  
I. Putnam in \cite{Putnam1} (cf. \cite{KamPutSpiel}, \cite{Putnam2}, \cite{Putnam4}, \cite{PutSp}, 
\cite{Thomsen}, etc. )
constructed various groupoids from Smale spaces and studied their  $C^*$-algebras. 
The class of Smale spaces contain two important subclasses of topological dynamical systems as its typical examples.
One is the class of shifts of finite type, which are sometimes called topological Markov shifts.
The other one is the class of hyperbolic toral automorphisms.
The study of the former class from the view point of $C^*$-algebras is closely related to the study of Cuntz-Krieger algebras as in 
\cite{Holton}, \cite{KamPut}, \cite{KilPut}, \cite{MaJOT2019}, etc.
That of the latter class 
is closely related to the study of the crossed product $C^*$-algebras 
of the homeomorphisms of the hyperbolic automorphisms on the torus.

In this paper, we will focus on the study of the latter class, the hyperbolic toral automorphisms
from the view points of $C^*$-algebras constructed from the associated groupoids as Smale spaces. 
Let $A = 
\begin{bmatrix}
a & b \\
c & d
\end{bmatrix}
 \in \GL (2,\Z)$ 
be a hyperbolic matrix. 
Let
$q: \mathbb{R}^2 \longrightarrow \mathbb{R}^2/\Z^2$
be the natural quotient map.
We denote by $\mathbb{R}^2/\Z^2$ the  two-dimensional torus $\T^2$
with metric $d$ defined by
$$
d( x,y) = \inf\{ \|  z-w \|  :  q(z) =x, q(w) = y, \, z,w \in \mathbb{R}^2\} \quad \text{ for } 
x,y \in \T^2
$$  
where 
$\| \cdot \| $ is the Euclid norm on $\mathbb{R}^2.$
Then the matrix $A$ defines a homeomorphism on $\T^2$
which is called a hyperbolic toral automorphism.
%The hyperbolic toral automorphisms consist of an imprtant and interesting  subclass
%of Smale spaces as well as irreducible shifts of finite type.   
It is a specific example of an Anosov diffeomorphism on a compact Riemannian manifold
(see \cite{Bowen2}, \cite{Smale}, etc.). 
Let $\lambda_u, \lambda_s$ be the eigenvalues of $A$ such that 
$|\lambda_u| > 1 > |\lambda_s|.$ 
They are both real numbers.
Let $v_u =(u_1, u_2), v_s =(s_1,s_2)$ 
be the normalized eigenvectors for $\lambda_u, \lambda_s$, respectively.
The direction along $v_u$ expands by $A$, 
whereas the direction of $v_s$ expands by $A^{-1}$.  
These directions determine local product structure which makes $\T^2$ a Smale space.
The groupoid $G_A^a$ introduced by 
D. Ruelle \cite{Ruelle1} of the asymptotic equivalence relation is defined by 
\begin{equation}
G_A^a =\{ (x, z) \in \T^2 \times \T^2 \mid 
\lim_{n\to{\infty}}d(A^nx, A^nz) = \lim_{n\to{\infty}}d(A^{-n}x, A^{-n}z) =0\} 
\end{equation}
with its unit space
\begin{equation}
(G_A^a)^{(0)} =\{ (x, x) \in \T^2 \times \T^2 \} =\T^2. 
\end{equation}
The multiplication and the inverse operation on $G_A^a$ 
are defined by 
$$
(x,z)(z,w) = (x,w), \qquad (x,z)^{-1} = (z,x) \qquad 
\text{ for } 
(x,z), (z,w) \in G_A^a.
$$
As in \cite{Putnam1}, the groupoid $G_A^a$ has a natural topology defined by inductive limit topology, 
which makes $G_A^a$ \'etale.
The \'etale groupoid $G_A^a$ is called the asymptotic groupoid 
for the hyperbolic toral automorphism $(\T^2, A)$.
We will first see that 
the groupoid $G_A^a$ is realized as a transformation groupoid 
$\T^2\rtimes_{\alpha^A}\Z^2$ 
by a certain action
$\alpha^A: \Z^2 \longrightarrow \Homeo(\T^2)$
associated to $G_A^a$, so that 
the $C^*$-algebra $C^*(G_A^a)$ of  the groupoid $G_A^a$
is isomorphic to the $C^*$-algebra of the crossed product 
$C(\T^2) \rtimes_{\alpha^A} \Z^2$ by the induced action 
$\alpha^A:\Z^2\longrightarrow \Aut(C(\T^2))$.
As the action
$\alpha^A: \Z^2 \longrightarrow \Homeo(\T^2)$
is free and minimal
having a unique invariant ergodic measure, 
a general theory of $C^*$-crossed product ensures that
$C(\T^2) \rtimes_{\alpha^A} \Z^2$ 
is a simple AT-algebra having a unique tracial state 
(cf. \cite{Phillips}, \cite{Putnam1}, \cite{PutSp} ).

Let 
$A = 
\begin{bmatrix}
a & b \\
c & d
\end{bmatrix}
 \in \GL (2,\Z)$ 
be a hyperbolic matrix which satisfies  $\det(A) = \pm 1,$
We denote by 
$\Delta(A) = (a+d)^2 - 4 (ad-bc)$ 
the discriminant of the characteristic polynomial of the matrix $A$, which is positive.
We will show the following result.
\begin{theorem} [{Theorem \ref{thm:main1}} and {Proposition \ref{prop:tauK}}] \label{thm:thm1.1}
The $C^*$-algebra $C^*(G_A^a)$ of the \'etale groupoid $G_A^a$  
for a hyperbolic matrix 
$A = 
\begin{bmatrix}
a & b \\
c & d
\end{bmatrix}
$
is a simple AT-algebra with unique tracial state $\tau$ 
that is isomorphic to the four-dimensional  non-commutative torus generated by four unitaries 
$U_1, U_2, V_1, V_2$ 
satisfying the following relations:
\begin{gather*}
U_1 U_2 = U_2 U_1,\qquad V_1 V_2 = V_2 V_1, \\
V_1 U_1 = e^{2\pi i \theta_1}U_1 V_1, \qquad V_1 U_2 = e^{2\pi i \theta_2}U_2 V_1, \\
V_2 U_1 = e^{2\pi i \theta_3}U_1 V_2, \qquad V_2 U_2 = e^{2\pi i \theta_4}U_2 V_2, 
\end{gather*}
where
\begin{equation*}
\theta_1=\frac{1}{2}( 1 + \frac{a-d}{\sqrt{\Delta(A)}}), \quad
\theta_2 =\frac{c}{\sqrt{\Delta(A)}}, \quad
\theta_3 =\frac{b}{\sqrt{\Delta(A)}}, \quad
\theta_4=\frac{1}{2}( 1 - \frac{a-d}{\sqrt{\Delta(A)}}).
\end{equation*}
%\begin{align*}
% &(\theta_1, \theta_2,\theta_3,\theta_4) \\
%=&
%\begin{cases}
%\left( \frac{1}{2}( 1 + \frac{|a-d|}{\sqrt{\Delta(A)}}),
%\frac{|a-d|}{a-d}\frac{c}{\sqrt{\Delta(A)}},
%\frac{|a-d|}{a-d}\frac{b}{\sqrt{\Delta(A)}},
%\frac{1}{2}( 1 - \frac{|a-d|}{\sqrt{\Delta(A)}} ) \right)  & \text{ or }\\
%\left( \frac{1}{2}( 1 - \frac{|a-d|}{\sqrt{\Delta(A)}}),
%-\frac{|a-d|}{a-d}\frac{c}{\sqrt{\Delta(A)}},
%-\frac{|a-d|}{a-d}\frac{b}{\sqrt{\Delta(A)}},
%\frac{1}{2}( 1 + \frac{|a-d|}{\sqrt{\Delta(A)}})\right)  & \text{ if  } a \ne d,\\
%(\frac{1}{2}, \frac{1}{2}\sqrt{\frac{c}{b}},\frac{1}{2}\sqrt{\frac{b}{c}}, \frac{1}{2}) 
%& \text{ or }\\ 
%(\frac{1}{2}, -\frac{1}{2}\sqrt{\frac{c}{b}},-\frac{1}{2}\sqrt{\frac{b}{c}}, \frac{1}{2}) 
%& \text{ if } a =d,
%\end{cases} 
%\end{align*}
The range $\tau_*(K_0(C^*(G_A^a)))$
of the tracial state $\tau$ of the $K_0$-group $K_0(C^*(G_A^a))$
of the $C^*$-algebra $C^*(G_A^a)$ is
\begin{equation}
\tau_*(K_0(C^*(G_A^a))) = \Z + \Z\theta_1 +\Z\theta_2 +\Z\theta_3 \quad\text{ in } \mathbb{R}. \label{eq:tauK0}
\end{equation}
\end{theorem}
We note that the slopes $\theta_i, i=1,2,3,4$
are determined by the formulas  
\eqref{eq: deftheta12}, \eqref{eq: deftheta34}
for the slopes of the eigenvectors 
$v_u =(u_1, u_2), v_s =(s_1, s_2)$ of the hyperbolic matrix $A$.

Since the \'etale groupoid $G_A^a$ is a flip cpnjugacy invariant
and the $C^*$-algebra $C^*(G_A^a)$ has a unique tracial state written $\tau$, 
% $$
% \tau_*(K_0(C^*(G_A^a))) =
%\tau_*(K_0(C^*(G_B^a))).
%$$
we know that 
the trace value $\tau_*(K_0(C^*(G_A^a)))$
is a flip conjugacy invariant of the hyperbolic toral automorphism
$(\T^2,A).$

As commuting matrices have common eigenvectors, 
we know that 
if two matrices $A, B \in \GL(2,\Z)$ commute with each other, 
then the $C^*$-algebras 
$C^*(G_A^a)$ and $C^*(G_B^a)$ are canonically isomorphic.
Hence two matrices 
$ 
\begin{bmatrix}
1 & 1 \\
1 & 0
\end{bmatrix}
$
and
$ 
\begin{bmatrix}
2 & 1 \\
1 & 1
\end{bmatrix}
$
have the isomorphic $C^*$-algebras.
On the other hand,  
as the range $\tau_*(K_0(C^*(G_A^a)))$ of the tracial state of the $K_0$-group
$K_0(C^*(G_A^a))$ is invariant under isomorphism class of the algebra
$C^*(G_A^a)$, 
the $C^*$-algebra
$C^*(G_{A_1})$ is not isomorphic to $C^*(G_{A_2})$
for the matrices
$A_1 = 
\begin{bmatrix}
1 & 1 \\
1 & 0
\end{bmatrix}
$
and
$A_2 = 
\begin{bmatrix}
3 & 1 \\
2 & 1
\end{bmatrix}
$
(Proposition \ref{prop:example}).

%We finally present  the asymptotic Ruelle algebra 
%$C^*(G_A^a\rtimes_A\Z)$ defined by the groupoid 
%\begin{equation*}
%G_A^a\rtimes_A\Z =\{ (x,n,z) \in \T^2\times\Z\times\T^2\mid (A^n(x), z) \in G_A^a\}
%\end{equation*}
%as a $C^*$-algebra generated by five unitaries satisfying certain commutation relations
%(Proposition \ref{prop:6.1}).

%%%%%%%%%%%%%%%%%%%%%%%%%%%%%%%%%%%%%%
%%%%%%%%%%%%%%%%%%%%%%%%%
\section{The groupoid $G_A^a$ and its $C^*$-algebra  $C^*(G_A^a)$}
%%%%%%%%%%%%%%%%%%%%%%%%%%%%%%%%%%%%%%%%%%%
%%%%%%%%%%%%%%%%%%%
For a vector $(m,n) \in \mathbb{R}^2$,
we write 
the vector
$(m,n)^t$
as
$\begin{bmatrix}
m\\
n
\end{bmatrix}
$ 
and sometimes identify
$(m,n)$ with 
$\begin{bmatrix}
m\\
n
\end{bmatrix}.
$ 
A matrix 
$A = 
\begin{bmatrix}
a & b \\
c & d
\end{bmatrix}
 \in \GL (2,\Z)$ 
with $\det(A) = \pm 1$ is said to be hyperbolic if
$A$ does not have eigenvalues of modulus $1$.
Let $\lambda_u, \lambda_s$ be the eigenvalues of $A$ such that 
$|\lambda_u| > 1 > |\lambda_s|.$ 
They are eigenvalues for unstable direction, stable direction, respectively.   
We note that  $b \ne 0, c \ne 0$ because of the conditions 
$ad-bc =\pm 1$ and $|\lambda_u| > 1 > |\lambda_s|.$  
Take nonzero eigenvectors $v_u, v_s$
for the eigenvalues  $\lambda_u, \lambda_s$ such that 
$\|v_u \| = \| v_s \| = 1.$
We set 
$v_u = (u_1, u_2), v_s = (s_1, s_2) \in \T^2$ as vectors.
The numbers $\lambda_u, \lambda_s, u_1, u_2, s_1, s_2$ 
are all real numbers
because of the hyperbolicity of the matrix $A$.
It is easy to see that the slopes 
$\frac{u_1}{u_2},\frac{s_1}{s_2}$ are irrational. 
We set
\begin{equation*}
r_A := \langle v_u | v_s \rangle.
\end{equation*}
Define two vectors
\begin{equation*}
v_1 := v_u -r_A v_s,\qquad
v_2 := r_A v_u -v_s.
\end{equation*}

\begin{lemma} For two vectors $x, z \in \T^2$, the following three conditions are equivalent.
\begin{enumerate}
\renewcommand{\theenumi}{\roman{enumi}}
\renewcommand{\labelenumi}{\textup{(\theenumi)}}
\item $(x,z) \in G_A^a.$
\item $ z = x + \frac{1}{1-r_A^2}\langle (m,n)|v_1\rangle v_u$
for some $m,n\in \Z.$
\item $ z = x + \frac{1}{1-r_A^2}\langle (m,n)|v_2\rangle v_s$
for some $m,n\in \Z.$
\end{enumerate}
\end{lemma}
\begin{proof}
For two vectors $x, z \in \T^2$
regarding them as elements of $\mathbb{R}^2$ modulo $\Z^2,$
we have $(x,z) \in G_A^a$  if and only if
\begin{equation}
z \equiv x + t v_u \equiv x + s v_s\quad  (\mod\,\,\, \Z^2) \qquad
\text{ for some } t,s \in \mathbb{R}. \label{eq:zx}
\end{equation}
In this case, we see that  $tv_u - s v_s = (m,n) $ for some $m,n\in \Z$
so that 
\begin{align}
\langle t v_u - s v_s  | v_u \rangle & = \langle (m,n) | v_u \rangle, \label{eq:tvu}\\ 
\langle t v_u - s v_s  | v_s \rangle & = \langle (m,n) | v_s \rangle \label{eq:tvs}
\end{align}
and 
we have 
\begin{equation}
t = \frac{1}{1-r_A^2}\langle (m,n)|v_1\rangle, \qquad 
s = \frac{1}{1-r_A^2}\langle (m,n)|v_2\rangle. \label{eq:tsmn} 
\end{equation}
This shows the implications
(i) $\Longrightarrow$ (ii) and  (iii).

Assume that (ii) holds.
By putting 
$s = \frac{1}{1-r_A^2}\langle (m,n)|v_2\rangle,$
we have the equalities both \eqref{eq:tvu} and \eqref{eq:tvs}, 
so that 
$tv_u - s v_s = (m,n) $.
Hence the equality \eqref{eq:zx} holds 
and we see that $(x,z)$ belongs to the groupoid $G_A^a.$
This shows that the implication 
(ii) $\Longrightarrow$ (i) holds, and similarly
(iii) $\Longrightarrow$ (i) holds.
\end{proof}
Let us define an action $\alpha^A: \Z^2 \longrightarrow \Homeo(\T^2)$  
in the following way.
We set
\begin{equation*}
\alpha^A_{(m,n)}(x) :=  x + \frac{1}{1-r_A^2}\langle (m,n)|v_1\rangle v_u, \qquad
(m,n) \in \Z^2, \, x \in \T^2.
\end{equation*}
For a fixed $(m,n) \in \Z^2,$
the map 
$x \in \T^2 \longrightarrow 
\alpha^A_{(m,n)}(x) \in \T^2$ 
is the parallel transformation along the vector  
$\frac{1}{1-r_A^2}\langle (m,n)|v_1\rangle v_u.$
Hence 
$\alpha^A_{(m,n)}$ defines a homeomorphism on the torus $\T^2$.
It is clear to see that 
$\alpha^A_{(m,n)}\circ \alpha^A_{(k,l)} = \alpha^A_{(m+k,n+l)}$ for
$(m,n), (k,l) \in \Z^2$.
\begin{lemma}\label{lem:freeminimal}
Keep the above notation. \hspace{6cm}
\begin{enumerate}
\renewcommand{\theenumi}{\roman{enumi}}
\renewcommand{\labelenumi}{\textup{(\theenumi)}}
\item 
If $\alpha^A_{(m,n)}(x) = x$ for some $x \in \T^2$, then $(m,n) =(0,0)$.
Hence the action $\alpha^A: \Z^2 \longrightarrow \Homeo(\T^2)$
is free.
\item
For $x \in \T^2$,
the set $\{\alpha^A_{(m,n)}(x) | \in (m,n) \in \Z^2\}$ is dense in $\T^2$.
Hence the action $\alpha^A: \Z^2 \longrightarrow \Homeo(\T^2)$
is minimal.
\end{enumerate}
\end{lemma}
\begin{proof}
(i) Suppose that 
$\alpha^A_{(m,n)}(x) = x$ for some $x \in \T^2,$
so that
$\frac{1}{1-r_A^2}\langle (m,n)|v_1\rangle v_u = (k,l)$
for some $(k,l) \in \Z^2.$
As the slope of the vector
$v_u$ is irrational, we have 
$(k,l) = (0,0)$ and hence $(m,n) =(0,0)$.

(ii) Let
$v_1 = (\gamma_1,\gamma_2).$
As the slope of $v_s$ is irrational and
$\langle v_s | v_1\rangle =0,$
the slope $\frac{\gamma_1}{\gamma_2}$ 
 of $v_1$ is irrational, so that 
the set $\{ m\gamma_1 + n \gamma_2 | m,n \in \Z\}$
is dense in $\mathbb{R}.$
Since
$\langle (m,n)  | v_1\rangle v_u  =  (m\gamma_1 + n \gamma_2 )v_u$
and the set 
$\{ x + t v_u \in \T^2 | t \in \mathbb{R}\}$ is dense in $\T^2,$
we see that 
the set 
\begin{equation*}
\{ x + \frac{1}{1-r_A^2}\langle (m,n)|v_1\rangle v_u \mid (m,n) \in \Z^2 \}
\end{equation*}
is dense in $\T^2.$
\end{proof}
The action $\alpha^A:\Z^2\longrightarrow \Homeo(\T^2)$
induces an action of $\Z^2$
to the automorphism group $\Aut(C(\T^2))$
of $C(\T^2)$ by
$f \in C(\T^2) \longrightarrow f\circ\alpha^A_{(m,n)} \in C(\T^2).
$
We write it still $\alpha^A$
without confusing.

If a discrete group $\Gamma$ acts freely on a compact Hausdorff space $X$
by an action $\alpha:\Gamma \longrightarrow \Homeo(X),$
the set 
$\{(x,\alpha_\gamma(x)) \in X\times X | x \in X, \gamma \in \Gamma\}$
has a groupoid structure in a natural way
(cf.  \cite{AnanRenault}, \cite{Renault1},  \cite{Renault2}). 
The groupoid is called a transformation groupoid written 
$X\rtimes_\alpha\Gamma$.  
\begin{proposition}
The \'etale groupoid $G_A^a$ is isomorphic to the transformation groupoid
$$
\T^2 \rtimes_{\alpha^A}\Z^2 
=\{ (x, \alpha^A_{(m,n)}(x)) \in \T^2\times \T^2 |
(m,n) \in \Z^2\}
$$
defined by the action $\alpha^A:\Z^2\longrightarrow \Homeo(\T^2).$
Hence the $C^*$-algebra $C^*(G_A^a)$ of the groupoid $G_A^a$ is isomorphic to the 
crossed product $C(\T^2)\rtimes_{\alpha^A}\Z^2$ of $C(\T^2)$ by the action
$\alpha^A$ of $\Z^2.$   
\end{proposition}
\begin{proof}
By the preceding discussions, a pair $(x,z) \in \T^2$ 
belongs to the groupoid $G_A^a$ 
if and only if 
$z = \alpha^A_{(m,n)}(x)$ for some $(m,n) \in \Z^2.$
Since the action 
$\alpha^A: \Z^2 \longrightarrow \Homeo(\T^2)$
is free, the groupoid
$G_A^a$ is identified with the transformation groupoid
$
\T^2 \rtimes_{\alpha^A}\Z^2 
$
in a natural way.
%\begin{equation*}
%(x, \alpha_{(m,n)}(x)) \in G_A^a \longrightarrow 
%(x, (m,n),\alpha_{(m,n)}(x))  \in \T^2 \rtimes_{\alpha}\Z^2 
%\end{equation*}
%gives rise to an isomorphism of the groupoids.
By a general theory of the $C^*$-algebras of groupoids
(\cite{AnanRenault}, \cite{Renault1}), 
the $C^*$-algebra
$
C^*(\T^2 \rtimes_{\alpha^A}\Z^2) 
$
of the groupoid 
$
\T^2 \rtimes_{\alpha^A}\Z^2 
$
is isomorphic to the crossed product
$
C^*(\T^2) \rtimes_{\alpha^A}\Z^2. 
$
\end{proof}
\begin{remark}\label{re:2.4}
Define a map $\alpha_A: \Z^2\longrightarrow \T^2$ by
\begin{equation}
\alpha_A(m,n) :=  \frac{1}{1-r_A^2}\langle (m,n)|v_1\rangle v_u, \qquad
(m,n) \in \Z^2.
\end{equation}
It is easy to see that the  \'etale groupoid $G_A^a$ may be written
\begin{equation}
G_A^a = \T^2 \rtimes \alpha_A(\Z^2) \label{eq:groupoidaa}
\end{equation}
as a transformation groupoid.
\end{remark}
We set
\begin{gather}
\theta_1 : = \frac{u_1 s_2}{u_1 s_2 - u_2 s_1}, \qquad
\theta_2 : = \frac{u_2 s_2}{u_1 s_2 - u_2 s_1}, \label{eq: deftheta12}\\
\theta_3 : = \frac{-u_1 s_1}{u_1 s_2 - u_2 s_1}, \qquad
\theta_4 : = \frac{-u_2 s_1}{u_1 s_2 - u_2 s_1}.\label{eq: deftheta34}
\end{gather}
\begin{lemma}\label{lem:thetazeta}
The real numbers $\theta_i, i=1,2,3,4$ satisfy
\begin{gather}
\frac{\theta_2}{\theta_1} = \frac{\theta_4}{\theta_3} = \frac{u_2}{u_1},\qquad
\frac{\theta_1}{\theta_3} = \frac{\theta_2}{\theta_4} = -\frac{s_2}{s_1}, \label{eq:thetas}\\
\theta_1 + \theta_4 = 1. \label{eq:theta14}
\end{gather}
Conversely, if real numbers 
$\zeta_i, i=1,2,3,4$ satisfy
\begin{gather}
\frac{\zeta_2}{\zeta_1} = \frac{\zeta_4}{\zeta_3} = \frac{u_2}{u_1},\qquad
\frac{\zeta_1}{\zeta_3} = \frac{\zeta_2}{\zeta_4} = -\frac{s_2}{s_1}, \label{eq:zetas}\\
\zeta_1 + \zeta_4 = 1 \label{eq:zeta14},
\end{gather}
then we have $\zeta_i = \theta_i, i=1,2,3,4.$
\end{lemma}
\begin{proof}
The identities \eqref{eq:thetas} and \eqref{eq:theta14} are immediate.
Conversely, suppose that 
real numbers 
$\zeta_i, i=1,2,3,4$ satisfy
\eqref{eq:zetas} and \eqref{eq:zeta14}.
As $\zeta_1 = \frac{u_2}{u_1}\zeta_2 =\frac{u_2}{u_1}( -\frac{s_2}{s_1}) \zeta_4,$
the equality \eqref{eq:zeta14} implies
$$
\{ \frac{u_2}{u_1}( -\frac{s_2}{s_1})  +1 \}\zeta_4 =1,
$$
 so that 
$$
\zeta_4  = \frac{-u_2 s_1}{u_1 s_2 - u_2 s_1}
$$
and hence 
\begin{equation*}
\zeta_1  = \frac{u_1 s_2}{u_1 s_2 - u_2 s_1}, \qquad
\zeta_2  = \frac{u_2 s_2}{u_1 s_2 - u_2 s_1}, \qquad
\zeta_3  = \frac{-u_1 s_1}{u_1 s_2 - u_2 s_1}.
\end{equation*}
\end{proof}

\begin{proposition}\label{prop:alpha1001}
For $x =(x_1, x_2) \in \T^2$, we have 
\begin{equation*}
\alpha^A_{(1,0)}(x_1,x_2) = (x_1+\theta_1, x_2+\theta_2), \qquad
\alpha^A_{(0,1)}(x_1,x_2) = (x_1+\theta_3, x_2+\theta_4),
\end{equation*}
and hence 
\begin{equation*}
\alpha^A_{(m,n)}(x_1,x_2) = (x_1+m\theta_1+n\theta_3, x_2+m\theta_2 +n\theta_4)
\quad
\text{ for }
(m,n) \in \Z^2.
\end{equation*}
\end{proposition}
\begin{proof}
We have
\begin{align*}
\alpha^A_{(m,n)}(x_1,x_2) 
= & (x_1,x_2) + \frac{1}{1-r_A^2}\langle (m,n)| v_u - r_A v_s\rangle v_u \\
= & (x_1,x_2) + \frac{1}{1-r_A^2}\langle (m,n)| (u_1 -r_A s_1, u_2 - r_A s_2)\rangle (u_1, u_2).
\end{align*}
In particular, for $(m,n) = (1,0), (0,1)$, we have 
\begin{align*}
\alpha^A_{(1,0)}(x_1,x_2) 
= & (x_1 + \frac{1}{1-r_A^2}(u_1 -r_A s_1)u_1, \,
      x_2 + \frac{1}{1-r_A^2}(u_1 -r_A s_1)u_2), \\
\alpha^A_{(0,1)}(x_1,x_2) 
= & (x_1 + \frac{1}{1-r_A^2}(u_2 -r_A s_2)u_1, \,
      x_2 + \frac{1}{1-r_A^2}(u_2 -r_A s_2)u_2).
\end{align*}
We put
$\xi_i = \frac{1}{1-r_A^2}(u_i -r_A s_i)$ for $i=1,2$ so that 
\begin{align}
\alpha^A_{(1,0)}(x_1,x_2) 
= & (x_1 + \xi_1 u_1, \,  x_2 + \xi_1 u_2), \label{eq:alpha10} \\
\alpha^A_{(0,1)}(x_1,x_2) 
= & (x_1 + \xi_2u_1, \, x_2 + \xi_2 u_2). \label{eq:alpha01}
\end{align}
We then have 
\begin{align*}
\xi_1 & = \frac{1}{1-r_A^2}\{u_1 -(u_1 s_1 + u_2 s_2) s_1\} 
          = \frac{1}{1-r_A^2}\{u_1( 1- s_1^2) -  u_2 s_2 s_1\} \\
       & %= \frac{1}{1-r_A^2}\{u_1s_2^2) -  u_2 s_2 s_1\}
          = \frac{1}{1-r_A^2} (u_1s_2 -  u_2 s_1) s_2 
\end{align*}
and  similarly 
\begin{align*}
\xi_2 & = \frac{1}{1-r_A^2}\{u_2 -(u_1 s_1 + u_2 s_2) s_2\} 
          = \frac{1}{1-r_A^2}\{u_2( 1- s_2^2) -  u_1 s_1 s_2\} \\
       & = \frac{1}{1-r_A^2} (u_2s_1 -  u_1 s_2) s_1. 
\end{align*}
Hence we have 
$\frac{\xi_1}{\xi_2} = - \frac{s_2}{s_1}.$
We also have 
\begin{align*}
\xi_1 u_1 + \xi_2 u_2 
= & \frac{1}{1-r_A^2}\{(u_1 - r_A s_1)u_1 +  (u_2 - r_A s_2) u_2 \} \\  
= & \frac{1}{1-r_A^2}\{u_1^2 + u_2^2 -  r_A(u_1  s_1 + u_2 s_2)  \} \\  
= & \frac{1}{1-r_A^2}( 1 - r_A^2)  =1.
\end{align*}
By Lemma \ref{lem:thetazeta}, we have
$\xi_1 u_1 = \theta_1, \, \xi_1 u_2 = \theta_2, \, \xi_2 u_1 = \theta_3, \, 
\xi_2 u_2 = \theta_4,$
proving  the desired assertion from the identities \eqref{eq:alpha10} and \eqref{eq:alpha01}.
\end{proof}
%%%%%%%%%%%%%%%%%%%%%%%%%%%%%%%%%%%
We will next express $\theta_i, i=1,2,3,4$ in terms of the matrix elements $a,b,c,d$ of $A$.
\begin{lemma}\label{lem:abcdlambda} The following identities hold.
 \begin{enumerate}
\renewcommand{\theenumi}{\roman{enumi}}
\renewcommand{\labelenumi}{\textup{(\theenumi)}}
\item 
\begin{gather*}
a\theta_1 + b \theta_2 = \lambda_u \theta_1, \qquad
a\theta_3 + b \theta_4 = \lambda_u \theta_3, \\
c\theta_1 + d \theta_2 = \lambda_u \theta_2, \qquad
c\theta_3 + d \theta_4 = \lambda_u \theta_4,
\end{gather*}
and hence 
\begin{equation}
a\theta_1 + b \theta_2 + c\theta_3 + d \theta_4 = \lambda_u. \label{eq:abcdu}
\end{equation} 
\item 
\begin{gather*}
a\theta_3 - b \theta_1 = \lambda_s \theta_3, \qquad
a\theta_4 - b \theta_2 = \lambda_s \theta_4, \\
c\theta_3 - d \theta_1 = -\lambda_s \theta_1, \qquad
c\theta_4 - d \theta_2 = -\lambda_s \theta_2,
\end{gather*}
and hence 
\begin{equation*}
a\theta_4 - b \theta_2 - c\theta_3 + d \theta_1 = \lambda_s. \label{eq:abcds}
\end{equation*} 
\end{enumerate}
\end{lemma}
\begin{proof}
By the identities
\begin{gather*}
{\begin{bmatrix}
\theta_1\\
\theta_2
\end{bmatrix}}
=\frac{ s_2}{u_1 s_2 - u_2 s_1}
{\begin{bmatrix}
u_1\\
u_2
\end{bmatrix}}, \qquad
{\begin{bmatrix}
\theta_3\\
\theta_4
\end{bmatrix}}
=\frac{ -s_1}{u_1 s_2 - u_2 s_1}
{\begin{bmatrix}
u_1\\
u_2
\end{bmatrix}}, \\
{\begin{bmatrix}
\theta_3\\
-\theta_1
\end{bmatrix}}
=\frac{ -u_1}{u_1 s_2 - u_2 s_1}
{\begin{bmatrix}
s_1\\
s_2
\end{bmatrix}}, \qquad
{\begin{bmatrix}
\theta_4\\
-\theta_2
\end{bmatrix}}
=\frac{ -u_2}{u_1 s_2 - u_2 s_1}
{\begin{bmatrix}
s_1\\
s_2
\end{bmatrix}},
\end{gather*}
with $\theta_1 + \theta_4 =1$, we see the desired assertions.
\end{proof}
\begin{lemma}\hspace{6cm}
\begin{enumerate}
\renewcommand{\theenumi}{\roman{enumi}}
\renewcommand{\labelenumi}{\textup{(\theenumi)}}
\item 
$(a\theta_1 + b \theta_2) \theta_4 = (c\theta_3 + d \theta_4) \theta_1.$
\item 
$(a\theta_3 - b \theta_1) \theta_2 = (-c\theta_4 + d \theta_2) \theta_3.$
\end{enumerate}
Hence we have 
\begin{equation*}
b \theta_2 = c \theta_3.
\end{equation*}
\end{lemma}
\begin{proof}
(i) By the first and the fourth identities in Lemma \ref{lem:abcdlambda} (i),
we know the identity (i). 
The identities of (ii) is similarly shown to those of (i).
By (i) and (ii) with the identity $\theta_1 \theta_4 = \theta_2 \theta_3$,
we get
$b \theta_2 = c \theta_3.
$
\end{proof}
Recall that $\Delta(A)$ denotes the discriminant 
$(a + d)^2 - 4(a d -b c)$
of the characteristic polynomial
of the matrix $A$.
The real number $\Delta(A)$ is positive because of the hyperbolicity of $A$.
By elementary calculations, we see the following lemma.
\begin{lemma}
The identities
\begin{gather*}
\theta_1\cdot \theta_4 = \theta_2\cdot \theta_3, \qquad 
\theta_1 +\theta_4 = 1,\\
(a\theta_1 + b \theta_2) \theta_4 = (c\theta_3 + d \theta_4) \theta_1, \qquad
(a\theta_3 - b \theta_1) \theta_2 = (-c\theta_4 + d \theta_2) \theta_3
\end{gather*}
imply
\begin{align}
 &(\theta_1, \theta_2,\theta_3,\theta_4) \\
=&
\begin{cases}
\left( \frac{1}{2}( 1 + \frac{|a-d|}{\sqrt{\Delta(A)}}),
\frac{|a-d|}{a-d}\frac{c}{\sqrt{\Delta(A)}},
\frac{|a-d|}{a-d}\frac{b}{\sqrt{\Delta(A)}},
\frac{1}{2}( 1 - \frac{|a-d|}{\sqrt{\Delta(A)}}) \right)  & \text{ or }\\
\left( \frac{1}{2}( 1 - \frac{|a-d|}{\sqrt{\Delta(A)}}),
-\frac{|a-d|}{a-d}\frac{c}{\sqrt{\Delta(A)}},
-\frac{|a-d|}{a-d}\frac{b}{\sqrt{\Delta(A)}},
\frac{1}{2}( 1 + \frac{|a-d|}{\sqrt{\Delta(A)}} ) \right)  & \text{ if  } a \ne d,\\
(\frac{1}{2}, \frac{1}{2}\sqrt{\frac{c}{b}},\frac{1}{2}\sqrt{\frac{b}{c}}, \frac{1}{2}) 
& \text{ or }\\ 
(\frac{1}{2}, -\frac{1}{2}\sqrt{\frac{c}{b}},-\frac{1}{2}\sqrt{\frac{b}{c}}, \frac{1}{2}) 
& \text{ if } a =d.
\end{cases} \label{eq:2.15}
\end{align}
\end{lemma}
We thus have the following theorem.
\begin{theorem}\label{thm:main1}
The $C^*$-algebra $C^*(G_A^a)$ of the groupoid $G_A^a$ for a hyperbolic matrix 
$A = 
\begin{bmatrix}
a & b \\
c & d
\end{bmatrix}
$
is isomorphic to the simple $C^*$-algebra  generated by four unitaries 
$U_1, U_2, V_1, V_2$ 
satisfying the following relations:
\begin{gather*}
U_1 U_2 = U_2 U_1,\qquad V_1 V_2 = V_2 V_1, \\
V_1 U_1 = e^{2\pi i \theta_1}U_1 V_1, \qquad V_1 U_2 = e^{2\pi i \theta_2}U_2 V_1, \\
V_2 U_1 = e^{2\pi i \theta_3}U_1 V_2, \qquad V_2 U_2 = e^{2\pi i \theta_4}U_2 V_2, 
\end{gather*}
where
\begin{equation}
\theta_1=\frac{1}{2}( 1 + \frac{a-d}{\sqrt{\Delta(A)}}), \quad
\theta_2 =\frac{c}{\sqrt{\Delta(A)}}, \quad
\theta_3 =\frac{b}{\sqrt{\Delta(A)}}, \quad
\theta_4=\frac{1}{2}( 1 - \frac{a-d}{\sqrt{\Delta(A)}}). \label{eq:2.18}
\end{equation}
Hence the $C^*$-algebra $C^*(G_A^a)$ is isomorphic to the four-dimensional non-commutative torus.
\end{theorem}
\begin{proof}
As in Lemma \ref{lem:freeminimal}, the action 
$\alpha^A: \Z^2 \longrightarrow \Homeo(\T^2)$
is free and minimal, hence the $C^*$-crossed product
$C(\T^2)\rtimes_{\alpha^A} \Z^2$ is simple.
The $C^*$-crossed product
is canonically identified with the $C^*$-crossed product
$((C(\T)\otimes C(\T)) \rtimes_{\alpha^A_{(1,0)}}\Z) \rtimes_{\alpha^A_{(0,1)}}\Z. $
 Let $U_1, U_2 $ be the unitaries in $C(\T)\otimes C(\T)$ defined by
$U_1(t,s) = e^{2\pi i t}, U_2(t,s) = e^{2\pi i s}$.
Let $V_1, V_2$ be the implementing unitaries corresponding to the automorphisms
$\alpha^A_{(1,0)}, \alpha^A_{(0,1)},$ respectively. 
By Proposition \ref{prop:alpha1001}, 
we know the commutation relations 
among the unitaries $U_1, U_2, V_1, V_2$
for the slopes $\theta_1, \theta_2, \theta_3, \theta_4$
satisfying \eqref{eq:2.15}.
The second values of \eqref{eq:2.15} go to the first of \eqref{eq:2.15}
by substituting $V_1, U_1 $ with $V_2, U_2,$
respectively.
The forth values of \eqref{eq:2.15} go to the third of \eqref{eq:2.15}
by substituting $V_1, U_1 $ with $V_1^*, U_1^*,$
respectively.
When $a=d$, we have $\Delta(A) = 4 bc >0$ so that 
$\pm\sqrt{\frac{c}{b}} = \frac{c}{\sqrt{\Delta(A)}}, 
 \pm\sqrt{\frac{b}{c}} = \frac{b}{\sqrt{\Delta(A)}}.
$
Hence the first two of \eqref{eq:2.15} include the second two of \eqref{eq:2.15},
so that we may unify \eqref{eq:2.15} into \eqref{eq:2.18}.
%%%% 
%Hence the $C^*$-algebra $C^*(G_A^a)$ is nothing but the four-dimensional non-commutative torus studied in \cite{Elliott}, \cite{Rieffel2}.
\end{proof}
Since the $C^*$-algebra $C^*(G_A^a)$ is isomorphic to
 a simple four-dimensional non-commutative torus,
  we know the following proposition by Slawny \cite{Slawny} (see also  Putnam \cite{Putnam1}).
\begin{proposition}[{Slawny \cite{Slawny}, Putnam \cite{Putnam1}}]
The $C^*$-algebra $C^*(G_A^a)$ has a unique tracial state.
\end{proposition}
\begin{remark}
\begin{enumerate}
\renewcommand{\theenumi}{\roman{enumi}}
\renewcommand{\labelenumi}{\textup{(\theenumi)}}
\item 
We note that the simplicity of the algebra  $C^*(G_A^a)$
comes from a general theory of Smale space $C^*$-algebras as in 
\cite{Putnam1}, \cite{PutSp}
as well as a unique existence of tracial state on it.
It also follows from a general theory of crossed product $C^*$-algebras  
because the action $\alpha^A$ of $\Z^2$ to $\Homeo(\T^2)$ is free and minimal.
It has been shown that a simple higher dimensional non-commutative torus is an AT-algebra by Phillips \cite{Phillips}. 
\item
Suppose that 
 two hyperbolic matrices $A, B \in \GL(2,\Z)$ commute each other.
By \eqref{eq:groupoidaa}, the equality
$\alpha_A(\Z^2) = \alpha_B(\Z^2)$ 
holds for the commuting matrices $A$ and $B$,
because they have the same eigenvectors. 
Hence we know that $G_A^a =G_B^a$, 
so that the $C^*$-algebras 
$C^*(G_A^a)$ and $C^*(G_B^a)$ are isomorphic.
\end{enumerate}
\end{remark}

%%%%%%%%%%%%%%%%%%%%%%%%%%%%%%%%%%%%%%%%%%
%%%%%%%%%%%%%%%%%%%%%%%%%%%%%%%%%%%%%%%%%%%%%%%%%%%%%%
\section{The range  $\tau_*(K_0(C^*(G_A^a)))$}
%%%%%%%%%%%%%%%%%%%%%%%%%%%%%%%%%%%%%%%%%%%%%%%%%
%%%%%%%%%%%%%%%%%%%%%%%%%%%%%%%%%%%%%%%%%%%%%%%%%
In this section, we will describe the trace values 
$\tau_*(K_0(C^*(G_A^a)))$ of the $K_0$-group of the $C^*$-algebra $C^*(C_A^a)$
in terms of the hyperbolic matrix $A$.

In \cite{Rieffel}, M. A. Rieffel studied K-theory for irrational rotation $C^*$-algebras
$A_\theta$ with irrational numbers $\theta$,
which are called two-dimensional non-commutative tori, and proved that 
 $\tau_*(K_0(A_\theta)) = \Z + \Z\theta$ in $\mathbb{R},$
where $\tau$ is the unique tracial state on $A_\theta$. 
In \cite{Elliott}, G. A. Elliott ( cf. \cite{Boca}, \cite{Phillips}, \cite{Rieffel2}, \cite{Slawny}, etc.) 
initiated to study higher-dimensional non-commutative tori. 
It is well-known that the $K$-groups of the four-dimensional non-commutative torus 
are computed in \cite{Elliott}
such as 
%the K-groups
%$K_i(C(\T^2)\rtimes_\alpha \Z^2)$ of the $C^*$-algebra
%$C(\T^2)\rtimes_\alpha \Z^2$ are
\begin{equation*}
K_0(C(\T^2)\rtimes_{\alpha^A} \Z^2) 
\cong 
K_1(C(\T^2)\rtimes_{\alpha^A} \Z^2) 
\cong \Z^8  \qquad (\cite{Elliott}, cf. \cite{Slawny}).
\end{equation*} 
For $g =(a_1, b_1, a_2, b_2), h =(c_1, d_1, c_2, d_2) \in \Z^4,$
we define a wedge product  $g\wedge h \in \Z^4$  by
\begin{equation*}
(a_1, b_1, a_2, b_2) \wedge (c_1, d_1, c_2, d_2) 
= \left( 
\begin{vmatrix}
a_1 & c_1 \\
b_1 & d_1
\end{vmatrix},
\begin{vmatrix}
a_1 & c_1 \\
b_2 & d_2
\end{vmatrix},
\begin{vmatrix}
a_2 & c_2 \\
b_1 & d_1
\end{vmatrix},
\begin{vmatrix}
a_2 & c_2 \\
b_2 & d_2
\end{vmatrix}
\right)
\end{equation*} 
where 
$
\begin{vmatrix}
x & y \\
z & w
\end{vmatrix}
= 
xw - yz.
$
Let $\Theta =[\theta_{jk}]_{j,k=1}^4$ be a $4\times 4$ skew symmetric matrix over $\mathbb{R}.$
We regard the matrix $\Theta$ as a linear map from $\Z^4\wedge\Z^4$ 
to $\mathbb{R}$ by defining
$\Theta(x\wedge y)= \Theta x \cdot y. $  
Then 
$\Theta\wedge\Theta:(\Z^4\wedge\Z^4)\wedge(\Z^4\wedge\Z^4) 
= \wedge^4 \Z^4 \longrightarrow \mathbb{R}$
is defined by  
$$
(\Theta\wedge\Theta) (x_1\wedge x_2)\wedge(x_3\wedge x_4) =
\frac{1}{2!2!} \sum_{\sigma\in \frak{S}_4}
\sgn(\sigma) \Theta(x_{\sigma(1)}\wedge x_{\sigma(2)})\Theta(x_{\sigma(3)}\wedge x_{\sigma(4)})
$$
for
$ x_1, x_2, x_3, x_4 \in \Z^4.$
Although we may generally define 
$\wedge^n \Theta : \wedge^{2n}\Z^4 \longrightarrow \mathbb{R},$
the wedge product 
$\wedge^{2n}\Z^4 =0$ for $n >3,$ 
so that 
$$
\exp_\wedge(\Theta) = 1 \oplus \Theta \oplus \frac{1}{2}(\Theta\wedge\Theta)
\oplus \frac{1}{6}(\Theta\wedge\Theta\wedge\Theta) \oplus \cdots \,  : \, 
\wedge^{\operatorname{even}} \Z^4 \longrightarrow \mathbb{R}
$$
becomes 
$$
\exp_\wedge(\Theta) = 1 \oplus \Theta \oplus \frac{1}{2}(\Theta\wedge\Theta).
$$
Let $A_{\Theta}$ be the universal $C^*$-algebra generated by four unitaries
$u_j, j=1,2,3,4$ subject to the commutation relations
$u_j u_k = e^{2 \pi i \theta_{jk}}u_k u_j,\,  j,k=1,2,3,4.$
The $C^*$-algebra $A_\Theta$ is called the four-dimensional non-commutative torus
(\cite{Elliott}).
If $\Theta$ is non-degenerate, the algebra $A_\Theta$
has a unique tracial state written $\tau.$
By Elliott's result in \cite{Elliott}, 
there exists an isomorphism
$h: K_0(A_{\Theta}) \longrightarrow  \wedge^{\operatorname{even}} \Z^4$
such that 
$\exp_\wedge(\Theta) \circ h = \tau_*,$
so that we have
\begin{equation}
\exp_\wedge(\Theta)(\wedge^{\operatorname{even}} \Z^4) = \tau_*(K_0(A_\Theta)). \label{eq:expTheta}
\end{equation}
\begin{proposition}\label{prop:tauK}
Let $\tau$ be the unique tracial state on $C^*(G_A^a)$.
Let $\theta_1, \theta_2, \theta_3$ be real numbers defined by 
\eqref{eq:2.18}.
Then we have 
\begin{equation}
\tau_*(K_0(C^*(G_A^a))) = \Z + \Z\theta_1 +\Z\theta_2 +\Z\theta_3 \quad\text{ in } \mathbb{R}. 
\label{eq:tracevalue}
\end{equation}
\end{proposition}
\begin{proof}
Take the unitaries 
$U_1, U_2, V_1, V_2$ and the real number $\theta_4$ together with
$\theta_1, \theta_2, \theta_3$ as in
Theorem \ref{thm:main1}.
We set the real numbers 
$\theta_{jk}, j,k=1,2,3,4$ such as 
$\theta_{jj} = \theta_{12} =\theta_{21} =\theta_{34} =\theta_{43} =0$ 
for $j=1,2,3,4$
and
$\theta_{13} = \theta_4,\,
 \theta_{14} = \theta_3,\,
 \theta_{23} = \theta_2,\,
\theta_{24} = \theta_1.
$ 
Let 
$u_1 = V_2, u_2 = V_1, u_3 = U_2, u_4 = U_1$
so that we have the commutation relations 
$$
u_j u_k = e^{2\pi i \theta_{jk}} u_k u_j, \qquad j,k = 1,2,3,4.
$$
%Put the $4 \times 4$ matrix $\Theta =[\theta_{jk}]_{j,k=1}^4$
%so that we have
%$\begin{bmatrix}
%0 & 0 & \theta_4 & \theta_3 \\
%0 & 0 & \theta_2 & \theta_1 \\
%- \theta_4 & -\theta_2 & 0 & 0\\
%- \theta_3 & -\theta_1 & 0 & 0.
%\end{bmatrix}$
As
$\theta_1 \cdot \theta_4 = \theta_2\cdot\theta_3,$
we have 
\begin{equation*}
\theta_{12} \theta_{34} -\theta_{13}\theta_{24} + \theta_{14}\theta_{23}
= 0.
\end{equation*}
By \eqref{eq:expTheta} or \cite{Elliott} (cf. \cite[2.21]{Boca}, \cite[Theorem 3.9]{Phillips}),
we have
\begin{align*}
\tau_*(K_0(C^*(G_A^a)))
= & \Z + \Z(\theta_{12} \theta_{34} -\theta_{13}\theta_{24} + \theta_{14}\theta_{23}
) + \sum_{1\le j < k\le 4}\Z\theta_{jk} \\
%= & \Z + \Z\theta_1+ \Z\theta_2+ \Z\theta_3+ \Z\theta_4 \\
= & \Z + \Z\theta_1+ \Z\theta_2+ \Z\theta_3.
\end{align*}
\end{proof}
\begin{remark}
Suppose that two hyperbolic toral automorphisms $(\T^2, A)$ and $(\T^2,B)$ are
topologically conjugate.
We then know that both the $C^*$-algebras $C^*(G_A^a)$ and $C^*(G_B^a)$ are isomorphic.
Since they have   unique tracial states $\tau_A$ and $\tau_B$ respectively,
we see that 
 $$
 \tau_{A*}(K_0(C^*(G_A^a))) =
\tau_{B*}(K_0(C^*(G_B^a))).
$$
We may also find a matrix $M \in GL(2,\Z)$ such that 
$AM = MB$ by \cite{AP}.
We then directly see that the ranges
$
 \tau_{A*}(K_0(C^*(G_A^a)))
 $ and $
\tau_{B*}(K_0(C^*(G_B^a)))$
coincide by using the formula
\eqref{eq:tracevalue}.
Similarly we may directly show that the equality 
$
 \tau_{A*}(K_0(C^*(G_A^a)))
=
 \tau_{A^{-1}*}(K_0(C^*(G_{A^{-1}}^a)))
 $
by the formula \eqref{eq:tracevalue}. 
\end{remark}

%%%%%%%%%%%%%%%%%%%%%%%%%%%%%%%%%%%%%%%%%%%%%%%%
%%%%%%%%%%%%%%%%%%%%%%%%%%%%%%%%%%%%%%%%%%%%%%%
\section{Examples}
%%%%%%%%%%%%%%%%%%%%%%%%%%%%%%%%%
In this section, we will present some examples. 

{\bf 1.}
$A = 
\begin{bmatrix}
1 & 1 \\
1 & 0
\end{bmatrix}.
$
Since $a =b=c= 1, d=0$, we have by Theorem \ref{thm:main1},
\begin{equation}
(\theta_1, \theta_2,\theta_3,\theta_4)
=(\frac{1}{2}(1+\frac{1}{\sqrt{5}}), \frac{1}{\sqrt{5}}, \frac{1}{\sqrt{5}}, 
\frac{1}{2}(5-\frac{1}{\sqrt{5}}). \label{eq:theta1234}
\end{equation}
It is easy to see that 
\begin{equation*}
\tau_*(K_0(C^*(G_A^a)))
=  \Z + \frac{5 + \sqrt{5}}{10}\Z.
\end{equation*}

\begin{proposition}
Let $A$ be the matrix 
$
\begin{bmatrix}
1 & 1 \\
1 & 0
\end{bmatrix}.
$
Put
$\theta = \frac{1}{2}(1+\frac{1}{\sqrt{5}}).$
Then the $C^*$-algebra
$C^*(G_A^a)$ is isomorphic to the tensor product $A_{\theta}\otimes A_{5\theta}$
between the irrational rotation $C^*$-algebras  
$A_\theta$ and $A_{5\theta}$
with its rotation angles $\theta$ and $5\theta$ respectively.
\end{proposition}
\begin{proof}
Let $U_1, U_2, V_1, V_2$ be the generating unitaries in Theorem \ref{thm:main1}.
Since
$$
(\theta_1, \theta_2,\theta_3,\theta_4)
=(\theta, 2\theta -1,2\theta -1,1-\theta)
$$
by \eqref{eq:theta1234},
we have 
\begin{gather*}
U_1 U_2 = U_2 U_1,\qquad V_1 V_2 = V_2 V_1, \\
V_1 U_1 = e^{2\pi i \theta}U_1 V_1, \qquad V_1 U_2 = e^{2\pi i 2 \theta}U_2 V_1, \\
V_2 U_1 = e^{2\pi i 2\theta}U_1 V_2, \qquad V_2 U_2 = e^{-2\pi i \theta}U_2 V_2, 
\end{gather*}
We set
$$
u_1 = U_1 U_2^2, \qquad
u_2 = U_2, \qquad
v_1 = V_1 V_2^2, \qquad
v_2 =V_2.
$$
It is straightforward to see that the following equalities hold
\begin{gather*}
u_1 u_2 = u_2 u_1,\qquad v_1 v_2 = v_2 v_1, \\
v_1 u_1 = e^{2\pi i 5\theta}u_1 v_1, \qquad v_1 u_2 = u_2 v_1, \\
v_2 u_1 = u_1 v_2, \qquad v_2 u_2 = e^{-2\pi i \theta}u_2 v_2. 
\end{gather*}
Since the $C^*$-algebra $C^*(u_1, u_2, v_1, v_2)$ 
generated by $u_1, u_2, v_1, v_2$
coincides with $C^*(G_A^a),$
we have 
\begin{equation*}
C^*(G_A^a)\cong C^*(u_1, v_1) \otimes C^*(u_2, v_2) 
               \cong A_{5\theta} \otimes A_\theta. 
\end{equation*}
\end{proof}

\medskip

{\bf 2.}
$A = 
\begin{bmatrix}
3 & 1 \\
2 & 1
\end{bmatrix}.
$
Since $a =3, b=d= 1, d=2$, we have by Theorem \ref{thm:main1},
$$
(\theta_1, \theta_2,\theta_3,\theta_4)
=(\frac{3+\sqrt{3}}{6}, \frac{\sqrt{3}}{3}, \frac{\sqrt{3}}{6}, 
\frac{3-\sqrt{3}}{6})
$$
and
$$
\lambda_u = a \theta_1 + b \theta_2 + c\theta_3 + d \theta_4
= 2 + \sqrt{3}, \qquad
\lambda_s = a \theta_4 - b \theta_2 - c\theta_3 + d \theta_1
= 2 - \sqrt{3}.
$$
Since $\theta_4 = 1 -\theta_1$, 
$\theta_2 =  2 \theta_3,$
$\theta_1 = \frac{1}{2} + \theta_3$, 
the formula \eqref{eq:tracevalue} says that  
\begin{equation*}
\tau_*(K_0(C^*(G_A^a)))
=\Z + \Z\theta_1 + \Z \theta_2 + \Z \theta_3
= \frac{1}{2}\Z +  \frac{\sqrt{3}}{6}\Z.
\end{equation*}
\begin{proposition}\label{prop:example}
Let
$A_1 = 
\begin{bmatrix}
1 & 1 \\
1 & 0
\end{bmatrix}
$
and
$A_2 = 
\begin{bmatrix}
3 & 1 \\
2 & 1
\end{bmatrix}.
$
Then the $C^*$-algebra
$C^*(G_{A_1}^a)$ is not isomorphic to $C^*(G_{A_2}^a).$
\end{proposition}
\begin{proof}
Since the algebra $C^*(G_A^a)$ has the  unique tracial state $\tau$,
the range $\tau_*(K_0(C^*(G_A^a)))$ of $\tau$ of the $K_0$-group $K_0(C^*(G_A^a))$
is invariant under isomorphism class of the $C^*$-algebra.    
As
$$
\tau_*(K_0(C^*(G_{A_1}^a))) = \Z +  \frac{5+\sqrt{5}}{10}\Z, \qquad
\tau_*(K_0(C^*(G_{A_1}^a)))=\frac{1}{2}\Z +  \frac{\sqrt{3}}{6}\Z,
$$
we see that
$
\tau_*(K_0(C^*(G_{A_1}^a))) \ne
\tau_*(K_0(C^*(G_{A_2}^a))),
$
so that 
the $C^*$-algebra
$C^*(G_{A_1})$ is not isomorphic to $C^*(G_{A_2}).$
\end{proof}

%%%%%%%%%%%%%%%%%%%%%%%%%%%%%%%%%%%%%%%%%%%

%%%%%%%%%%%%%%%%%%%%%%
{\it Acknowledgments:}
%The author would like to deeply thank 
%Ian F. Putnam for his comments and suggestions on the earlier version of the paper.
This work was supported by JSPS KAKENHI Grant Numbers 15K04896, 19K03537.

%%%%%%%%%%%%%%%%%%%%%%%%%%%%%%%%%%%%%%%%%%%%%%%%%
%%%%%%%%%%%%%%%%%%%%%%%%%%%%%%%%%%%%%%%%%%%%%%%%%%

\end{document}